\documentclass[11pt]{amsart}
\usepackage{palatino}
\newcommand{\FF}{\mathbb{F}}

\newcommand{\NN}{\mathbb{N}}

\newcommand{\ZZ}{\mathbb{Z}}

\newcommand{\cA}{\mathcal{A}}

\newcommand{\cD}{\mathcal{D}}

\newcommand{\fb}{\mathfrak{b}}

\newcommand{\fg}{\mathfrak{g}}

\newcommand{\fu}{\mathfrak{u}}

\DeclareMathOperator{\ad}{ad}

\DeclareMathOperator{\Char}{char}

\DeclareMathOperator{\Der}{Der}

\DeclareMathOperator{\Hom}{Hom}

\DeclareMathOperator{\Span}{Span}

\theoremstyle{plain}
\numberwithin{equation}{section}
\newtheorem{Theorem}{Theorem}[section]
\newtheorem{Lemma}[Theorem]{Lemma}
\newtheorem{Corollary}[Theorem]{Corollary}

\theoremstyle{Theorem}

\theoremstyle{remark}
\newtheorem*{Remark}{Remark}

\numberwithin{equation}{section}

\usepackage{pictex}
\usepackage{amscd}
\usepackage{latexsym}
\usepackage{amssymb}
\usepackage{eucal}
\usepackage{eufrak}
\usepackage{verbatim}
\usepackage{hyperref}
\begin{document}
\title[Tensor products]{Tensor products of the defining representations over the Witt algebra in positive characteristic}
\author{Hao Chang \lowercase{and} Yu-Feng Yao}
\address[Hao Chang]{School of Mathematics and Statistics, Central China Normal University, 430079 Wuhan, People's Republic of China}
\email{chang@ccnu.edu.cn}
\address[Yu-Feng Yao]{Department of Mathematics, Shanghai Maritime University, 201306 Shanghai, People's Republic of China}
 \email{yfyao@shmtu.edu.cn}



\thanks{This work is supported by National Natural Science Foundation of China (Grant Nos. 11801204, 11771279 and 12071136).}

\begin{abstract}
Let $A(1):=k[X]/(X^p)$ be the natural representation of the Witt algebra $W(1)$ over an algebraically closed field of prime characteristic $p>3$.
In this note, we decompose the $W(1)$-module $A(1)\otimes A(1)$ into two invariant subspaces, and precisely construct their Jordan-H\"older composition series. As a consequence, we obtain all decomposition factors of the tensor product of the simple restricted $W(1)$-module with ``highest" weight $p-1$.
\end{abstract}
\maketitle

\section{Introduction}
The tensor product of representations is an important ingredient in the representation theory.
Let $(\fg,[p])$ be a restricted Lie algebra over an algebraically closed field $k$ of positive characteristic,
and $\fu(\fg)$ its restricted enveloping algebra.
Representations of $\fu(\fg)$ are, in general, no longer completely reducible, and the structures of
tensor products of irreducible representations can be extremely difficult to work out in detail even if we have determined all the irreducible representations.
In the case $\fg=\mathfrak{sl}_2$,
the decomposition of $L(\lambda)\otimes L(\mu)$, tensor products of two simple modules, is completely understood in \cite{Pre91} (see also \cite{BO}).
Let $W_1$ be the infinite-dimensional simple Lie algebra of Cartan type over $\mathbb{C}$.
In \cite{Ta}, Tanaka studied the $2$-fold tensor product of the defining
representation for $W_1$, denoted as $P_{(2)}$.
She decomposed $P_{(2)}$ into two invariant
subspaces and constructed Jordan-H\"older composition series of them.

Block and Wilson \cite{BW} showed that all restricted simple Lie algebras over $k$ are either classical or of Cartan type provided that the characteristic of $k$ is larger than $7$.
Lie algebras of Cartan type fall into four infinite classes of algebras: $W$, $S$, $H$ and $K$.
In this article, we consider the Witt algebra $\fg=W(1)$,
which was found by E. Witt as the first example of non-classical
simple Lie algebra in 1930s.
By definition, the Witt algebra is the full derivation algebra of
the truncated polynomial ring $A(1):=k[X]/(X^p)$, where $p$ is the characteristic of $k$.
In this paper,
we will give a complete structure of the $2$-fold tensor product $\otimes^2 A(1)$.
The purpose of this note is to obtain an analogue of the main result of \cite{Ta} by employing techniques
that work in positive characteristics.

\emph{Throughout this paper, $k$ denotes an algebraically closed field of characteristic $\Char(k)=:p>3$. All vector spaces are assumed to be finite-dimensional over $k$}.

\section{Notations and Preliminaries}

\subsection{The Witt algebra $W(1)$}\label{W1 definition}
Let $A(1):=k[X]/(X^p)$ be the truncated polynomial ring,
whose canonical generator will be denoted by $x:=X+(X^p)$.
The Lie algebra $W(1):=\Der(A(1))$ is called the \textit{Witt algebra}, see \cite[IV]{SF} for more details.
We let $\partial\in W(1)$ denote the derivative with respect to the variable $x$.
Then $\{\partial\}$ is a basis of the $A(1)$-module $W(1)$, so that $\dim_k W(1)=p$.

Throughout this paper, we always denote by $\fg$ the Witt algebra $W(1)$.

It is well-known that $\fg$ is a restricted simple Lie algebra,
and has a natural $\mathbb{Z}$-grading $\fg=\sum_{i=-1}^{p-2}\fg_{i}$,
where $\fg_{i}=ke_i,$ and $e_i:=x^{i+1}\partial$ for $-1\leq i\leq p-2$.
Associated with this grading,
one has the following natural filtration:
$$\fg=\fg_{(-1)}\supset\fg_{(0)}\supset\cdots\supset\fg_{(p-2)}\supset 0,$$
where $$\fg_{(i)}=\sum\limits_{j\geq i}\fg_{j}, \,-1\leq i\leq p-2.$$

A $\mathbb{Z}$-graded $\fg$-module is a $\fg$-module $V=\bigoplus\limits_{i\in\ZZ}V_i$
such that $\fg_i\cdot V_j\subseteq V_{i+j}$,
for any $i,j\in\ZZ$.
The space $A(1)$ has a natural $\mathbb{Z}$-grading as
\begin{equation}\label{grading}
A(1)=\bigoplus\limits_{j=0}^{p-1}A_j,~{\rm where}~A_j=kx^j~{\rm for}~0\leq j\leq p-1.
\end{equation}
The algebra $\fg$ acts naturally on the truncated polynomial ring $A(1)$ with $\fg_iA_j\subseteq A_{i+j}$.
It is a $\mathbb{Z}$-graded $\fg$-module,
and is called the \textit{defining representation} of $\fg$.

The \textit{restricted enveloping algebra} of $\fg$ is denoted by $\fu(\fg)$.
By definition,
$$\fu(\fg):=U(\fg)/(\{x^p-x^{[p]};~x\in\fg\})$$
is a finite-dimensional quotient of the ordinary enveloping algebra $U(\fg)$ (cf. \cite[(II.2.5)]{SF}).
The degree of a monomial $u:=e_{-1}^{j_{-1}}\cdot e_{0}^{j_0}\cdot e_{1}^{j_1}\cdot\cdots\cdot e_{p-2}^{j_{p-2}}\in\fu(\fg)$,
is defined as $d(u):=-1(j_{-1})+0j_0+1j_1+\cdots+(p-2)j_{p-2}$.
Let $\fu(\fg)_l$ be the linear span of monomials of degree $l$,
and we call an element $u\in\fu(\fg)_l$ a
\textit{homogeneous element} of degree $l$ in $\fu(\fg)$.
A homogeneous element $u\in\fu(\fg)_l$ acts homogeneously,
that is, $u$ maps a homogeneous component $A_k$ into $A_{k+l}$.

\begin{Remark}
When $p=2$, $W(1)$ is no longer simple,
and when $p=3$ it is isomorphic to $\mathfrak{sl}(2)$.
\end{Remark}

\subsection{Representations of $W(1)$}\label{representation W1}
Note that the $\mathbb{Z}$-graded
restricted Lie algebra has a standard triangular decomposition $\fg=\fg^{-}\oplus\fg_0\oplus\fg^{+}$ such that $\fg^{-}=\fg_{-1}=ke_{-1}$,
$\fg_0=ke_0$ and $\fg^{+}=\fg_{(1)}$.
Then $\fg^{+}$ is an ideal of the trigonalizable Lie algebra $\fb^+:=\fg_0\oplus\fg^{+}$.
In this paper,
we are only interested in those irreducible representations which are restricted, i.e.,
those irreducible $\fu(\fg)$-modules.

Recall that the simple $\fu(\fg)$-modules are well-known and were first determined by Chang \cite{Chang}.
To describe them,
we define a one-dimensional $\fb^+$-module $k_{\lambda}$ on which $\fg^+$ acts trivially and $e_0$ acts by multiplication via the scalar $\lambda\in\{0,\dots,p-1\}$.
Then one defines the corresponding Verma module $Z(\lambda):=\fu(\fg)\otimes_{\fu(\fb^+)}k_\lambda$.
It is easy to see that $Z(\lambda)$ is $p$-dimensional and has a basis $\{m_0,\dots,m_{p-1}\}$ where the action
of $\fg$ is given by
\begin{equation}\label{Nakano formula}
e_k.m_j=(j+k+1+(k+1)\lambda)m_{j+k}
\end{equation}
(see for instance \cite[(2.2.1)]{N92}).
The $Z(\lambda)$ are all simple except for $Z(0)$ and $Z(p-1)$.
The former has a trivial simple quotient and the latter has
a trivial submodule and $(p-1)$-dimensional simple quotient.
We denote the corresponding simple quotient modules by $L(\lambda)$ ($\lambda\in\{0,\dots,p-1\}$).
Direct computation shows that $A(1)\cong Z(p-1)$.

On the other hand, we may recognize the simple restricted module as the following:
Let $L=L(\lambda)$, there is, up to scalars, a unique vector, $m_0$ killed by $e_{-1}=\partial$.
By formula (\ref{Nakano formula}) $e_0.m_0=(\lambda+1)m_0$ whenever $L(\lambda)=Z(\lambda)$.
For the trivial module $L(0)$, of course $e_0$ has weight zero. One checks that for $L(p-1)$,
$e_0$ has weight $1$ on a vector killed by $e_{-1}$. Hence, the action of $e_{-1}$ and $e_0$ on $L$ determines $L$ up to isomorphism.
It should be notice that we may identify the adjoint module as $L(p-2)$: The element $e_{-1}$ is killed by $\ad e_{-1}$ and $[e_0, e_{-1}]=-e_{-1}$.
Thus $\lambda+1=-1$ modulo $p$ and so $\lambda=p-2$.

Let $\fb^-:=\fg_0\oplus\fg^{-}$. Similarly, we can define $L^-(\lambda)$ to be the simple head of $Z^-(\lambda):=\fu(\fg)\otimes_{\fu(\fb^-)}k_\lambda$.
By the above observation, we record the following fact:
\begin{Lemma}
The following statements hold:
\begin{eqnarray*}L(\lambda)=
\begin{cases}
L^-(\lambda), &\lambda=0,\cr
L^-(\lambda+1), &1\leq\lambda\leq p-2,\cr
L^-(1), &\lambda=p-1.
\end{cases}
\end{eqnarray*}
\end{Lemma}

\section{$\mathbb{Z}$-graded $W_1$-modules}
Let us consider the $2$-fold tensor product $A_{(2)}:=\otimes^2 A(1)$.
We can identify $A_{(2)}$ with the truncated polynomial ring $k[X_1,X_2]/(X_1^p,X_2^p)$,
whose canonical generators will be denoted by $x_1,x_2$.
The action of $e_j$ on $A_{(2)}$ is
\begin{equation}\label{action formula}
e_j\cdot f(x_1,x_2):=x_1^{j+1}\frac{\partial}{\partial x_1}f+x_2^{j+1}\frac{\partial}{\partial x_2}f.
\end{equation}

Note that $\cA:=A_{(2)}$ is $\mathbb{Z}$-graded by $\cA_i:=\langle x_1^{\alpha_1}x_2^{\alpha_2};~\alpha_1+\alpha_2=i\rangle$.
In view of (\ref{action formula}),
the module $\cA$ obtains the structure of a $\mathbb{Z}$-graded $\fg$-module.

\begin{Lemma}\label{ei injective}
Let $\cA=\cA_0\oplus\cA_1\oplus\cdots\oplus\cA_{2(p-1)}$ be the $\mathbb{Z}$-graded decomposition
and let $s\in\{1,2\}$ be given.
Then the operator $e_s:\cA_n\rightarrow\cA_{n+s}$ is injective whenever $n>0$ and $n+s<p$.
\end{Lemma}
\begin{proof}
Take an element $v\in\cA_n$ written as $v=\sum_{i=0}^na_ix_1^ix_2^{n-i}$ with coefficients $a_i\in k$,
we have
\begin{equation*}
e_s.v=\sum\limits_{i=0}^{n+s}\left((i-s)a_{i-s}+(n-i)a_i\right)x_1^ix_2^{n+s-i},
\end{equation*}
where we make the convention that $a_i=0$ for $i<0$ or $i>n$.
Suppose that $v\in\ker(e_s)$, then
\begin{equation*}
(i-s)a_{i-s}+(n-i)a_i=0~~~\mathrm{for}~~~0\leq i\leq n+s.
\end{equation*}
For $i=0$ we obtain $-sa_{-s}+na_0=0$, so that $a_0=0$.
We solve these equations successively, we obtain $a_i=0$ for $0\leq i\leq n$.
\end{proof}

Recall that $\fb^+=\fg_0\oplus\fg^+$.
Let $\fu(\fb^+)$ be the restricted enveloping algebra of $\fb^+$.
This is a $\mathbb{Z}$-graded subalgebra of $\fu(\fg)$, i.e.,
$$\fu(\fb^+)_l=\fu(\fb^+)\cap\fu(\fg)_l.$$
Moreover, $\fu(\fb^+)$ has a natural PBW-filtration:
\begin{equation}\label{PBW filtration}
k=\fu(\fb^+)^{(0)}\subseteq\fu(\fb^+)^{(1)}\subseteq\fu(\fb^+)^{(2)}\subseteq\cdots\subseteq\fu(\fb^+),
\end{equation}
where $\fu(\fb^+)^{(r)}:=\Span_k\{u=e_0^{r_0}e_1^{r_1}\cdots e_{p-2}^{r_{p-2}};~r_0+r_1+\cdots+r_{p-2}\leq r\}$.

Given an integer $i\in\mathbb{Z}$, we denote by $\overline{i}$ the corresponding element in $\{0,\dots,p-1\}$ such that
$\overline{i}\equiv i~\mathrm{mod}~p$.

Let $V\subseteq\cA$ be a $\mathbb{Z}$-graded $\fg$-submodule generated by an element $v_i\in\cA_i\cap\ker e_{-1}$.
It follows that $V=\fu(\fb^+).v_i$.
Clearly, the module $V$ is a {\it lowest weight module}.
Using (\ref{action formula}) one can show by direct computation that $v_i$ is a weight vector with weight $\overline{i}$,
the so-called {\it lowest weight} of $V$.

\begin{Lemma}\label{e-1 surjective}
Let $V\subseteq\cA$ be a $\mathbb{Z}$-graded $\fg$-submodule generated by an element $v_i\in\cA_i\cap\ker e_{-1}$ and
$V=V_i\oplus\cdots\oplus V_s$ be the $\mathbb{Z}$-graded decomposition.
If the lowest weight is not $0$,
then the operator $e_{-1}:V_{l+i+1}\rightarrow V_{l+i}$ is surjective whenever $0\leq l<p-2$.
\end{Lemma}
\begin{proof}
We show the assertion by induction on $l$. When $l=0$, since $V_i=kv_i$ and $v_i=e_{-1}(\frac{1}{2}e_1v_i)$, the assertion is true.
When $l=1$, $V_{1+i}=ke_1v_i$, and $e_1v_i=e_{-1}(\frac{1}{3}e_2v_i)$, the assertion is also true. In the following we assume that $l\geq 2$, and suppose that the assertion is true for $l-1$. We aim to show that the assertion is true for $l$. For that, note that $V_{l+i}=\fu(\fb^+)_l.v_i$.
We need to show that for any $u\in\fu(\fb^+)_l\cap\fu(\fb^+)^{(r)}$, there exists some $u'\in\fu(\fb^+)_{l+1}\cap\fu(\fb^+)^{(r)}$ such that
\begin{equation}\label{indunction 1}
u.v_i=e_{-1}(u'v_i)\in e_{-1}(V_{l+i+1})
\end{equation}
by induction on $l$ and $r$. Note that (\ref{indunction 1}) is true for $l=1$ and any $r$. Suppose (\ref{indunction 1}) holds for any integer less then $l$ and any $r$. Next we show that it also holds for $l$ and any $r$.  When $r=0$, since $l\geq 2$, we have $u=0$, (\ref{indunction 1}) is obvious true. When $r=1$,
it follows from direct computation that
$$(l+2)e_l.v_i=[e_{-1},e_{l+1}].v_i=e_{-1}e_{l+1}.v_i.$$
The assumption $l<p-2$ now yields $l+2\neq 0$.
Hence the assertion is true for $r=1$.

Assume that (\ref{indunction 1}) is true for some $r\geq 1$. Next we prove that
\begin{equation}\label{indunction j>=0}
e_j\tilde{u}.v_i\in e_{-1}(\fu(\fb^+)_{l+1}\cap\fu(\fb^+)^{(r+1)}).v_i
\end{equation}
for $\tilde{u}\in\fu(\fb^+)_{l-j}\cap\fu(\fb^+)^{(r)}$ by induction on $j\geq 0$.
When $j=0$, $e_0\tilde{u}.v_i=(l+i)\tilde{u}.v_i$. The induction hypothesis implies that (\ref{indunction j>=0}) is true for $j=0$.
Now assume $j\geq 1$, we assume that $e_{j-1}\tilde{u}'.v_i=e_{-1}(\hat{u}'v_i)$ for some $\hat{u}'\in\fu(\fb^+)_{l+1}\cap\fu(\fb^+)^{(r+1)}$, where
$\tilde{u}'\in\fu(\fb^+)_{l-j+1}\cap\fu(\fb^+)^{(r)}$ with $\tilde{u}.v_i=e_{-1}\tilde{u}'.v_i$,
then we have
\begin{eqnarray*}
e_j\tilde{u}.v_i&=&e_je_{-1}\tilde{u}'.v_i\\
&=&(e_{-1}e_j-(j+1)e_{j-1})\tilde{u}'.v_i\\
&=&e_{-1}(e_j\tilde{u}'-(j+1)\hat{u}').v_i\in e_{-1}(\fu(\fb^+)_{l+1}\cap\fu(\fb^+)^{(r+1)}).v_i.
\end{eqnarray*}
Thus the assertion (\ref{indunction j>=0}) holds for any $j$.
It follows that $u.v_i\in e_{-1}(V_{l+i+1})$ for all $u\in\fu(\fb^+)_l$,
as desired.
\end{proof}

\begin{Lemma}\label{subquotient}
Let $V\subseteq\cA$ be a $\mathbb{Z}$-graded $\fg$-submodule generated by an element $v_i\in\cA_i\cap\ker e_{-1}$.
Then $V$ has $L^{-}(\overline{i})$ as its quotient module.
\end{Lemma}
\begin{proof}
This follows directly from the universal property of $Z^{-}(\overline{i})$.
\end{proof}

\section{The module $A_{(2)}$}
The module $A_{(1)}=A(1)$ is nothing but $Z(p-1)$ as a restricted $\fg$-module (see Section \ref{representation W1}).
The module $A(1)$ has an invariant subspace $k$, where $k$ denotes the space of constant functions in $A(1)$.
The quotient module $A(1)/k$ is simple and $A(1)/k\cong L(p-1)$ as $\fg$-modules.

In this section,
we will consider the structure of the $2$-fold tensor product $\cA=A_{(2)}:=A(1)\otimes A(1)$.
\subsection{Direct sum decomposition}
Let $S_2$ be the symmetric group of degree $2$.
The group algebra $kS_2$ acts naturally on $\cA=A_{(2)}$ from the right.
Let $\cA_s$ and $\cA_a$ be the subspaces consisting of symmetric elements and anti-symmetric ones, respectively.
Our general assumption $p>3$ implies that $kS_2$ is semisimple,
and it commutes with $W(1)$.
\begin{Lemma}
Let $\fg:=W(1)$ and $\cA=A_{(2)}$.
Then the $\fg$-module $\cA$ is decomposed as $\cA=\cA_s\oplus \cA_a$ into two $\fg$-modules.
\end{Lemma}
\begin{proof}
Note that $kS_2$ is a semisimple algebra.
By general theory, we can decompose $\cA$ as a $kS_2$-module:
\begin{equation*}\label{S2 decomposition}
\cA=\sum\limits_{D\in\mathcal{Y}_2}\Hom_{S_2}(\sigma_D, \cA)\otimes\sigma_D,
\end{equation*}
where $\mathcal{Y}_2$ is the set of Young diagram of size $2$ and $\sigma_D$ is the corresponding irreducible representation of $S_2$.
Hence, the following isomorphism is obvious.
$$\cA_s\cong\Hom_{S_2}(\sigma_1,\cA)\otimes\sigma_1,~~~\cA_a\cong\Hom_{S_2}(\sigma_2,\cA)\otimes\sigma_2$$
as $\fg\times S_2$-modules, where $\sigma_1$ is the trivial representation and $\sigma_2$ is the sign representation.
\end{proof}

The space $A_{(2)}^+:=\otimes^2(A(1)/k)=\otimes^2(L(p-1))$ is called the {\it top level} of $A_{(2)}$.
By the same token, the top level $\cA^+=A_{(2)}^+$ is also decomposed as $\cA^+=\cA_s^+\oplus \cA_a^+$.

More precisely, we put
\begin{equation*}\label{bs prime}
\cA_s'=\Span_k\{x_1^i+x_2^i;~i=0,1,\cdots,p-1\},
\end{equation*}
\begin{equation*}\label{ba prime}
\cA_a'=\Span_k\{x_1^i-x_2^i;~i=1,2,\cdots,p-1\},
\end{equation*}
then $$\cA_s^+=\cA_s/\cA_s',~\cA_a^+=\cA_a/\cA_a'$$ and
$$\cA_s'\cong A(1)\cong Z(p-1),~\cA_a'\cong A(1)/k\cong L(p-1).$$

We record the following basic observation:
\begin{Lemma}
Dimensions of weight subspaces of these modules are as follows:
\begin{center}
\begin{tabular}{|c|c|c|c|c|c|}
\hline
weight   & $0$ ~~~& $1$~~~~& $2$~~~~~& $\cdots$~~~~~& $p-1$\\
\hline
$\cA$      & $p$    & $p$    & $p$ & $\cdots$&$p$\\
\hline
$\cA_s$    & $\frac{p+1}{2}$    & $\frac{p+1}{2}$     & $\frac{p+1}{2}$ & $\cdots$&$\frac{p+1}{2}$ \\
\hline
$\cA_a$    & $\frac{p-1}{2}$    & $\frac{p-1}{2}$     & $\frac{p-1}{2}$ & $\cdots$&$\frac{p-1}{2}$ \\
\hline
$\cA^+$    & $p-1$    & $p-2$    & $p-2$& $\cdots$&$p-2$\\
\hline
$\cA_s^+$  & $\frac{p-1}{2}$    & $\frac{p-1}{2}$     & $\frac{p-1}{2}$ & $\cdots$&$\frac{p-1}{2}$ \\
\hline
$\cA_a^+$  & $\frac{p-1}{2}$    & $\frac{p-3}{2}$     & $\frac{p-3}{2}$ & $\cdots$&$\frac{p-3}{2}$ \\
\hline
\end{tabular}
\end{center}
\end{Lemma}
\begin{proof}
Let $M$ be one of them with weight space decomposition
$$M=\sum\limits_{\lambda\in\FF_p}M_\lambda.$$
In view of Section \ref{representation W1}, all non-zero weight spaces of a restricted $\fg$-module have the same dimension.
Consequently,
$$\dim_k M=\dim_k M_0+(p-1)\dim_k M_{\lambda}$$
for any $0\neq\lambda\in\FF_p$.
Then our assertions follow from a direct computation.
\end{proof}

\subsection{Some key lemmas}
\begin{Lemma}\label{ker e-1 cap A_i}
Keep notations as before. Then we have
\begin{enumerate}
\item
$\dim_k((\cA_s^+)_i\cap\ker e_{-1})=\left\{
\begin{aligned}
1,~& i=2,4,\dots,p-1, \\
0,~& i=3,5,\dots,p.
\end{aligned}\right.$

\item
$\dim_k((\cA_a^+)_i\cap\ker e_{-1})=\left\{
\begin{aligned}
1,~& i=3,5,\dots,p, \\
0,~& i=4,\dots,p-1.
\end{aligned}\right.$
\end{enumerate}
\end{Lemma}
\begin{proof}
We only prove (1), as (2) can be treated similarly.
Let $v_i\in(\cA^+_s)_i$ be a homogeneous element.

Assume first that $i=2m$ is an even number.
If $m=1$, then $(\cA^+_s)_2=\Span_k\{x_1x_2\}$.
The assertion is true because $e_{-1}.(x_1x_2)=x_1+x_2=0$.
Hence we assume that $m>1$.
We write
$$v_i=v_{2m}=\sum\limits_{j=1}^{m-1}a_j(x_1^jx_2^{2m-j}+x_1^{2m-j}x_2^j)+a_mx_1^mx_2^m.$$
According to (\ref{action formula}), we have
\begin{eqnarray*}
e_{-1}.v_i&=&\sum\limits_{j=1}^{m-1}ja_{j}(x_1^{j-1}x_2^{2m-j}+x_1^{2m-j}x_2^{j-1})+ma_m(x_1^{m-1}x_2^m+x_1^mx_2^{m-1})\\
&+&\sum\limits_{j=1}^{m-1}(2m-j)a_{j}(x_1^jx_2^{2m-j-1}+x_1^{2m-j-1}x_2^{j})\\
&=&\sum\limits_{j=1}^{m-1}((2m-j)a_j+(j+1)a_{j+1})(x_1^jx_2^{2m-j-1}+x_1^{2m-j-1}x_2^{j}).
\end{eqnarray*}
Suppose that $v_i\in\ker(e_{-1})$.
Comparing coefficients yields
$$(2m-j)a_j+(j+1)a_{j+1}=0$$ for every $j\in\{1,\dots,m-1\}.$
Consequently, $\dim_k((\cA_s^+)_{2m}\cap\ker e_{-1})=1$.

When $i=2m+1$ is odd, we let
$$v_i=v_{2m+1}=\sum\limits_{j=1}^{m}a_j(x_1^jx_2^{2m+1-j}+x_1^{2m+1-j}x_2^j).$$
If $v_i\in\ker(e_{-1})$, then one can show by direct computation that
$$a_j(j-2m-1)=(j+1)a_{j+1}~{\rm for~every}~1\leq j\leq m-1,~~~~2a_m(m+1)=0.$$
Our assumption $i\leq p$ implies that $a_j=0$ for all $j$.
Therefore one has in this case $\dim_k((\cA_s^+)_i\cap\ker e_{-1})=0$.
\end{proof}
Now we denote by $v_{i,s}~(i=2,4,\dots,p-1)$ and $v_{j,a}~(j=3,5,\dots,p)$ the nonzero elements of $(\cA_s^+)_i\cap\ker e_{-1}$ and
$(\cA_a^+)_j\cap\ker e_{-1}$, respectively.
We let $\cA_s^+[i]:=\fu(\fg).v_{i,s}$ and $\cA_a^+[j]:=\fu(\fg).v_{j,a}$ the $\fg$-modules generated by $v_{i,s}$ and $v_{j,a}$,
respectively.

\begin{Lemma}\label{vi+2inAsi}
Given $i\in\{2,\dots,p-3\}$ and $j\in\{3,\dots,p-2\}$,
we have $v_{i+2,s}\in\cA_s^+[i]$ and $v_{j+2,a}\in\cA_a^+[j]$, respectively.
\end{Lemma}
\begin{proof}
For $(a,b)\in k^2$, we consider the homogeneous element
$$x_{a,b}:=(ae_1^2+be_2).v_{i,s}\in(\cA_s^+[i])_{i+2}.$$
As $i$ is an even number, we let $i=2m$ with $1\leq m\leq (p-3)/2$ and write
$$v_{i,s}=\sum\limits_{j=1}^{m-1}a_j(x_1^jx_2^{2m-j}+x_1^{2m-j}x_2^j)+a_mx_1^mx_2^m.$$
Direct computation shows that the coefficient of $x_1^{m+1}x_2^{m+1}$ in $x_{a,b}$ is
$$2am((m-1)a_{m-1}+ma_m)+2b(m-1)a_{m-1}.$$
Since $v_{i,s}\in\ker e_{-1}$, the proof of Lemma \ref{ker e-1 cap A_i} implies that $ma_{m}=-(m+1)a_{m-1}$.
It follows that
$$2am((m-1)a_{m-1}+ma_m)+2b(m-1)a_{m-1}=(2bm-4am-2b)a_{m-1}.$$
A similar argument yields that $e_{-1}.x_{a,b}=(8am+2a+3b)e_1.v_{i,s}$.
We denote by
$$U:=\{(a,b)\in k^2;~bm-2am-b\neq 0\}$$
and
$$V:=\{(a,b)\in k^2;~8am+2a+3b=0\},$$
respectively.
It is easy to check that $U\cap V\neq\emptyset$.
Thanks to Lemma \ref{ker e-1 cap A_i},
there exits an element $(a_0,b_0)\in U\cap V$ such that $v_{i+2,s}=x_{a_0,b_0}\in\cA_s^+[i]$,
as asserted.

One can argue similarly for another case.
\end{proof}

\begin{Lemma}\label{homogeneous space dim Lemma}
Let $v$ be one of $v_{i,s}~(i=2,4,\dots,p-1)$, $v_{i,a}~(i=3,5,\dots,p)$.
Then $\dim_k\fu(\fb^+)_l.v=\left[\frac{l}{2}\right]+1$ whenever $l+i\leq p$.
\end{Lemma}
\begin{proof}
As before, we only consider the case that $v=v_{i,s}$ for some $i\in\{2,4,\dots,p-1\}$.
We denote by $V:=\fu(\fb^+).v$ the submodule of $\cA_s^+$ generated by $v$.
Note that $V$ is $\mathbb{Z}$-graded,
$$V=V_i\oplus\cdots\oplus V_{2p-2}.$$
According to Lemma \ref{e-1 surjective}, there is a short exact sequence:
$$(\ast) \ \ \ \ \ \ \ (0)\rightarrow V_{l+i+1}\cap\ker e_{-1}\rightarrow V_{l+i+1}\stackrel{e_{-1}}{\rightarrow}V_{l+i}\rightarrow(0),$$
where $0\leq l<p-2$.

We prove $\dim_kV_{l+i}=\dim_k\fu(\fb^+)_l.v=\left[\frac{l}{2}\right]+1$ by induction on $l$.
Of course it is true for $l=0$.
We assume that the assertion is true for numbers less then $l$.
In the case where $l$ is an even number,
Lemma \ref{ker e-1 cap A_i} yields $\dim_k(V_{l+i}\cap\ker e_{-1})=1$.
We conclude from $(\ast)$ and the induction hypothesis
$$\dim_kV_{l+i}=\dim_kV_{l-1+i}+1=\left[\frac{l-1}{2}\right]+1+1=\left[\frac{l}{2}\right]+1.$$
In the case where $l$ is an odd number,
we apply Lemma \ref{ker e-1 cap A_i} to see that $\dim_k(V_{l+i}\cap\ker e_{-1})=0$.
By the same token, we have
$$\dim_kV_{l+i}=\dim_kV_{l-1+i}=\left[\frac{l-1}{2}\right]+1=\left[\frac{l}{2}\right]+1.$$
\end{proof}

Recall that the top level $\cA^+=\cA_s^+\oplus\cA_a^+$.
Let
\begin{equation*}\label{z-grading of As+}
\cA_s^+=(\cA_s^+)_2\oplus(\cA_s^+)_3\oplus\cdots\oplus(\cA_s^+)_{2p-2}
\end{equation*}
and
\begin{equation*}\label{z-grading of Aa+}
\cA_a^+=(\cA_a^+)_3\oplus(\cA_s^+)_4\oplus\cdots\oplus(\cA_s^+)_{2p-3}
\end{equation*}
be the $\mathbb{Z}$-graded decomposition of $\cA_s^+$ and $\cA_a^+$, respectively.

Clearly, one can show by direct computation that the elements $v_{2,s}:=x_1x_2\in(\cA_s^+)_2\cap\ker e_{-1}$ and $v_{3,a}:=x_1^2x_2-x_1x_2^2\in(\cA_a^+)_3\cap\ker e_{-1}$.
\begin{Lemma}\label{as=as2 and aa=aa3}
Keep notations as above. Then $\cA_s^+=\cA_s^+[2]$ and $\cA_a^+=\cA_a^+[3]$.
\end{Lemma}
\begin{proof}
We only prove the first equality and the other one follows similarly.
Setting $V:=\cA_s^+[2]=\fu(\fg).v_{2,s}$,
we consider the $\mathbb{Z}$-graded decomposition:
$$(\ast) \  \  \  \  \  \  \  \ V=V_2\oplus\cdots\oplus V_{2p-2}.$$
Thanks to Lemma \ref{homogeneous space dim Lemma},
we have
$$\dim_k V_p=\dim_k\fu(\fb^+)_{p-2}.v_{2,s}=\left[\frac{p-2}{2}\right]+1=\frac{p-1}{2}$$
and
$$\dim_k V_{p-1}=\dim_k\fu(\fb^+)_{p-3}.v_{2,s}=\left[\frac{p-3}{2}\right]+1=\frac{p-1}{2}.$$
We consider the weight space decomposition $V=\sum\limits_{\lambda\in\FF_p}V_\lambda$.
By $(\ast)$ we have $\dim_kV_{\overline{0}}=\dim_kV_p$ and $\dim_kV_{\overline{p-1}}=\dim_kV_{p-1}$.
Since all non-zero weight spaces of $V$ have the same dimension (Section \ref{representation W1}),
it follows that
$$\dim_kV=\dim_kV_p+(p-1)\dim_k V_{p-1}=\frac{p(p-1)}{2}=\dim_k\cA_s^+.$$
We thus obtain $\cA_s^+=\cA_s^+[2]$.
\end{proof}

\begin{Remark}
Suppose that $p=5$.
By a direct computation we have $(\cA_s^+)_6=\Span_k\{x_1^3x_2^3,x_1^4x_2^2+x_1^2x_2^4\}$ and $\dim_k(\cA_s^+)_6=2$,
so that Lemma \ref{homogeneous space dim Lemma} fails for $l+i>p$.
\end{Remark}

\subsection{Main result}
\begin{Theorem}\label{main theorem}
Let $\fg=W(1)$ and $\cA=A_{(2)}$. Then the following statements
hold:
\begin{itemize}
\item[(1)]
The top level $\cA_s^+:=\cA_s/Z(p-1)$ is an indecomposable $\fg$-module with the following composition series
\begin{equation}\label{main theorem formula 1}
\cA_s^+=\cA_s^+[2]\supseteq \cA_s^+[4]\supseteq\cdots\supseteq\cA_s^+[p-1]\supseteq(0).
\end{equation}
The composition factors are given by
\begin{equation}\label{main theorem formula 2}
\cA_s^+[i]/\cA_s^+[i+2]\cong L^{-}(\overline{i})\,\,\,\text{for}\,\,i\in\{2,4,\dots,p-1\}.
\end{equation}
\item[(2)]
The top level $\cA_a^+:=\cA_s/L(p-1)$ is an indecomposable $\fg$-module with the following composition series
\begin{equation}\label{main theorem formula 3}
\cA_a^+=\cA_a^+[3]\supseteq \cA_a^+[5]\supseteq\cdots\supseteq\cA_a^+[p]\supseteq(0).
\end{equation}
The composition factors are given by
\begin{equation}\label{main theorem formula 4}
\cA_a^+[j]/\cA_a^+[j+2]\cong L^{-}(\overline{j})\,\,\, \text{for}\,\,j\in\{3,5,\dots,p\}.
\end{equation}
\end{itemize}
\end{Theorem}
\begin{proof}
We have seen in Lemma \ref{as=as2 and aa=aa3} that $\cA_s^+=\cA_s^+[2]$.
Now we denote by $V:=\cA_s^+[4]$ the submodule generated by $v_{4,s}$ and consider the $\mathbb{Z}$-grading:
$$V=V_4\oplus\cdots\oplus V_{2p-2}.$$
By virtue of Lemma \ref{homogeneous space dim Lemma},
we have
$$\dim_k V_p=\dim_k\fu(\fb^+)_{p-4}.v_{4,s}=\left[\frac{p-4}{2}\right]+1=\frac{p-3}{2}$$
and
$$\dim_k V_{p-1}=\dim_k\fu(\fb^+)_{p-5}.v_{4,s}=\left[\frac{p-5}{2}\right]+1=\frac{p-3}{2}.$$
It follows that
$$\dim_kV=\dim_k V_p+(p-1)\dim_k V_{p-1}=\frac{p(p-3)}{2}.$$
Consequently, $V$ has codimension $p$ in $\cA_s^+$.

Note that $\cA_s^+=\cA_s^+[2]=\fu(\fg).v_{2,s}$ is a cyclic module.
Hence the quotient module $\cA_s^+/V$ is also generated by a lowest weight vector $\overline{v_{2,s}}=v_{2,s}+V$.
Using Lemma \ref{subquotient} and comparing dimension one obtains that the submodule $V$ is maximal and the quotient $\cA_s^+/V$ is isomorphic to $L^{-}(2)$.

Moreover,
Lemma \ref{vi+2inAsi} implies that
$$\dim_k(V_i\cap\ker e_{-1})=\left\{
\begin{aligned}
1,~& i=4,\dots,p-1, \\
0,~& i=5,\dots,p.
\end{aligned}\right.$$
We get (\ref{main theorem formula 1}) and (\ref{main theorem formula 2}) for all $i$ inductively.
This is (1).
One argues similarly to obtain (\ref{main theorem formula 3}) and (\ref{main theorem formula 4}).
\end{proof}

Now the $\fg$-module structure of $\cA=A_{(2)}$ is completely clear.
$\cA$ is decomposed into two invariant indecomposable subspaces as $\cA=\cA_s\oplus\cA_a$.
They contain submodules $\cA_a'$ and $\cA_a'$ respectively:
$$L(0)\cong k\subseteq Z(p-1)\cong\cA_s'\subseteq\cA_s, \  \  \  \  \  \  \  \ L(p-1)\cong\cA_a'\subseteq\cA_a.$$
The structure of quotient modules $\cA_s^+=\cA_s/\cA_s'$ and $\cA_a^+=\cA_a/\cA_a'$ can be obtained from Theorem \ref{main theorem}.

If $M$ is a $\fg$-module, then we denote
by $[M]$ the formal sum of all composition factors of $M$ in the Grothendieck group of the $\fg$-module category.
As a direct consequence of Theorem \ref{main theorem}, we have
\begin{Corollary}\label{comp factors}
Let $\fg=W(1)$ and $\cA=A_{(2)}$. Then the following statements
hold:
\begin{enumerate}
\item[(1)] $[\cA]=2[L(0)]+2[L(p-1)]+\sum\limits_{i=1}^{p-2} [L(i)].$
\item[(2)] $[L(p-1)\otimes L(p-1)]=\sum\limits_{j=0}^{p-2} [L(j)].$
\end{enumerate}
\end{Corollary}

\begin{Remark}
When $p=3$, $W(1)\cong\mathfrak{sl}(2)$.
Let $V(\lambda)$ denote the simple restricted $\mathfrak{sl}(2)$-module of highest weight $\lambda$ where
$0\leq \lambda\leq 2$.
Note that $V(1)$ is the $2$-dimensional natural representation of $\mathfrak{sl}(2)$.
In this case, the top level $\cA^+$ is isomorphic to $V(1)\otimes_k V(1)$.
Thanks to \cite[Lemma 2.3]{Pre91},
we have $V(1)\otimes_k V(1)\cong V(2)\oplus k$,
where $V(2)$ is isomorphic to the symmetric part of the top level and $k$ is isomorphic to the anti-symmetric one.
Hence Theorem \ref{main theorem} also holds for $p=3$.
\end{Remark}


\begin{thebibliography}{00}
\bibitem{BO}
G. Benkart, J. Osborn,
\textit{Representations of rank one Lie algebras of characteristic $p$}.
 Lect. Notes Math. \textbf{933} (1982), 1--37.

\bibitem{BW} R. Block and R. Wilson,
\textit{Classification of the restricted simple Lie algebras}.
J. Algebra \textbf{114} (1988), 115--259.

%\bibitem{BNW} B. Boe, D. Nakano and E. Wiesner,
%\textit{$\Ext^1$-quivers for the Witt algebra $W(1,1)$}.
%J. Algebra \textbf{322} (2009), 1548--1564.

\bibitem{Chang} H.-J. Chang,
\textit{\"Uber Wittsche Lie-Ringe}.
Abh. Math. Sem. Hansischen Univ. \textbf{14} (1941). 151--184.

\bibitem{Pre91} A. Premet ,
\textit{The Green ring of a simple three-dimensional Lie $p$-algebra}. (Russian)
Izv. Vyssh. Uchebn. Zaved. Mat. 1991, no. 10, 56--67; translation in
Soviet Math. (Iz. VUZ) \textbf{35} (1991), no. 10, 51--60.
%
%\bibitem{Ri} Kh. Rian,
%\textit{Extensions of the Witt algebra and applications}.
%J. Algebra Appl. \textbf{10} (2011), 1233--1259.
\bibitem{N92} D. K. Nakano,
\textit{Projective modules over Lie algebras of Cartan type}.
Mem. Amer. Math. Soc. \textbf{98} (470), 1992.

\bibitem{SF} H. Strade and R. Farnsteiner,
\textit{Modular Lie Algebras and their Representations}.
Pure and Applied Mathematics \textbf{116}. Marcel Dekker, 1988.

\bibitem{Ta} J. Tanaka,
\textit{Structure of tensor products of the defining representations of the Lie algebra $W_1$ of Cartan type}.
J. Math. Kyoto Univ. \textbf{41} (2001), no. 4, 795--807.
\end{thebibliography}
\end{document}